\documentclass {amsart}
\usepackage{amssymb,amsthm,color, amscd, mathtools}
\usepackage[full]{textcomp}

\title[Canonical decompositions of groups]{\bf Canonical decompositions of    abelian groups}

\author {Phill Schultz}
\address[phill.schultz@uwa.edu.au]{The University of Western Australia}

\theoremstyle{plain}
\newtheorem{theorem}{Theorem}[section]
\newtheorem{corollary}[theorem]{Corollary}
\newtheorem{lemma}[theorem]{Lemma}
\newtheorem{proposition}[theorem]{Proposition}

\theoremstyle{definition}

\theoremstyle{remark}
\newtheorem{remark} [theorem]{Remark}

\begin{document}

\renewcommand{\leq}{\leqslant}
\renewcommand{\geq}{\geqslant}
\newcommand{\Aut}{\mathop{\mathrm{Aut}}\nolimits}
\newcommand{\End}{\mathop{\mathrm{End}}\nolimits}
\newcommand{\Ker}{\mathop{\mathrm{Ker}}\nolimits}
\newcommand{\Hom}{\mathop{\mathrm{Hom}}\nolimits}
\newcommand{\Jo}{J\'onsson\ } 
\newcommand{\beglem}{\begin{lemma}\label{}}
\newcommand{\enlem}{\end{lemma}}
\newcommand{\begprop}{\begin{proposition}\label{}}
\newcommand{\enprop}{\end{proposition}}
\newcommand{\begpr}{\begin{proof}}
\newcommand{\enpr}{\end{proof}}

\newcommand\bbQ{{\mathbb{Q}}}
\newcommand\bbZ{{\mathbb{Z}}}
\newcommand\bbN{{\mathbb{N}}}
\newcommand\bbT{{\mathbb{T}}}
\newcommand{\bbP}{\mathbb{P}}

\newcommand{\calQ}{{\mathcal Q}}
\newcommand{\calP}{\mathcal{P}}
\newcommand{\calC}{\mathcal{C}}
\newcommand{\calD}{\mathcal{D}}
\newcommand{\calA}{\mathcal{A}}
\newcommand{\calH}{\mathcal{H}}
\newcommand{\calS}{\mathcal{S}}
\newcommand{\calB}{\mathcal{B}}
\newcommand{\calX}{\mathcal{X}}
\newcommand{\calY}{\mathcal{Y}}
\newcommand{\calL}{\mathcal{L}}
\newcommand{\calM}{\mathcal{M}}
\newcommand{\calG}{\mathcal{G}}
\newcommand{\calT}{\mathcal{T}}
\newcommand{\calI}{\mathcal{I}}
\newcommand{\calJ}{\mathcal{J}}
\newcommand{\calE}{\mathcal{E}}
\newcommand{\calK}{\mathcal{K}}
\newcommand{\calF}{\mathcal{F}}
\newcommand{\calV}{\mathcal{V}}
\newcommand{\calU}{\mathcal{U}}
\newcommand{\calCD}{\mathcal{CD}}
\newcommand{\calR}{\mathcal{R}}

 \newcommand{\ov}{\overline}
\newcommand{\rank}{\operatorname{rank}}
\newcommand{\type}{\operatorname{type}}
\newcommand{\lcm}{\operatorname{lcm}}
\newcommand{\Jon}{\operatorname{Jon}}
\newcommand{\Add}{\operatorname{Add}}

\newcommand{\la}{\langle} 
\newcommand{\ra}{\rangle}

\begin{abstract}
\noindent
Every torsion--free abelian group of finite rank has two essentially unique complete
direct decompositions whose summands come from specific classes of groups. 
\end{abstract} 
 \maketitle

\section{Introduction}   

An intriguing feature of abelian group theory is that torsion--free groups, even those of finite rank, may fail to have unique complete decompositions; see for example
 \cite[Chapter  12, \S5]{Fuchs} or  \cite{MS2018}).  Therefore it is interesting to show  that  every such group has   essentially unique complete decompositions with summands from certain identifiable classes. For example, it was shown in \cite{MS2018} that every finite rank torsion--free group $G$ has  a Main Decomposition,  $G=G_{cd}\oplus G_{cl}$ where $G_{cd}$ is completely decomposable and $G_{cl}$ has no rank 1 direct summand; $G_{cd}$ is unique up to isomorphism and $G_{cl}$ unique up to near isomorphism.
  
 Let $\calG$ be the class of  torsion--free abelian groups of finite rank and let $G\in\calG$. The aim of this paper is to show that:  
 \begin{itemize}\item   $G=G_{sd}\oplus G_{ni}$ where   $G_{sd}$ is a direct sum of  strongly indecomposable  groups and   $G_{ni}$ has no strongly indecomposable direct summand.  
  $G_{sd}$ is unique up to isomorphism and $G_{ni}$ up to near isomorphism.
 \item  $G=G^{(fq)}\oplus G^{(rq)}\oplus G^{(dq)}$, where $G^{(fq)}$ is an extension of a  completely decomposable group by a finite group,   
 $G^{(rq)}$ is an extension of a completely decomposable group by an infinite reduced torsion group, and 
 $G^{(dq)}$ is an extension of a completely decomposable group by an infinite torsion group which has a divisible summand. 
Each of the three summands is unique up to isomorphism. 
   \end{itemize}
  
  The first case, dealt with in Section 3, I call    a \textit{Principal Decomposition}. Since rank 1 groups are   strongly indecomposable, a Principal Decomposition is a refinement of a Main Decomposition;  the second case, the subject of Section 5,  I call a \textit{Torsion Quotient Decomposition}.  Of course the Principal Decomposition can be applied to each summand of the Torsion  Quotient Decomposition and \textit{vice versa}, leading to    decompositions of the form $\bigoplus G^{(i)}_j$, and it does not matter in which order the canonical decompositions are applied.

 Although this paper is intended to address   particular problems in abelian group theory, in Section 6, I show that the results can be couched in categorical form, so they have immediate applications to larger classes of modules, for example torsion--free finite rank modules over an integral domain.

\section{Notation}

Unless otherwise signalled, the word \lq group\rq\  denotes torsion--free abelian group of finite rank. Every   group   is isomorphic to a  subgroup of the infinite dimensional  rational vector space $V=\bbQ^{\bbN}$;  define $\calG$ to be the set of finite rank subgroups of $V$.   If $S\subseteq V,\ \la S\ra$ denotes the additive subgroup generated by $S,\ [S]$ the subspace generated by $S$ and $S^G_*$ the group $[S]\cap G$. Thus $[G]$ is the subspace of $V$ generated by $G$ and if $S\subseteq G$ then $S^G_*$ is the pure subgroup of $G$ generated by $S$.  When there is no ambiguity, we just write $S_*$, and we adopt the usual convention on singletons, namely $b_*=\la b\ra_*$ for $b\in G$.

Apart from side--stepping foundational problems,  these conventions have  the   advantage   that homomorphisms $G\to H$ can be identified with their unique extensions $[G]\to [H]$. In particular,  groups $G$ and $H$ are isomorphic if and only if there is an automorphism of $V$ mapping $G$ onto $H$, so the automorphism group $\Aut(G)$ coincides with the the stabilizer $\Aut_G([G])$ of the non--singular linear transformations of $[G]$. If $r\in\bbQ^*$, the non--zero rationals, then $rG\in\calG$. Multiplication by $r=a/b$ is an automorphism of $[G]$ which is an automorphism of $G$ when $G$ is divisible by  $b$. By $r=a/b$,  we intend that $0\ne a\in\bbZ$ and $0\ne b\in\bbN$ with $\gcd\{a,\,b\}=1$.
Another advantage of our conventions is that integrally independent subsets of $G$ are rationally independent subsets of $[G]$ contained in $G$, so we may omit the first adjective. The \textit{rank} of $G$ is the cardinality of some, and hence every,  maximal independent subset.

 Since every group is a direct sum of a unique divisible and a reduced group,  we limit consideration to reduced groups. A decomposition of  $G$ into non--zero direct summands is called \textit{complete} if these summands are indecomposable. Since $\rank(G)$ is finite,   $G$ always has complete decompositions, in general not unique even up to isomorphism. 

A rank 1 group is a group isomorphic to a non--zero subgroup of the rationals $\bbQ$, a completely decomposable (cd) group is a direct sum of rank 1  groups, and an  almost completely decomposable (acd) group is one  containing a completely decomposable group as a subgroup of finite index.

Groups $G$ and $H$ in $\calG$ are \textit{quasi--equal}, denoted $G\dot=H$, if there exists $r \in\bbQ^*$ such that $rG=H$,   and \textit{quasi--isomorphic}, denoted  $G\approx H$,  if $G'\cong H'$ for some $G'\dot=G$ and $H'\dot= H$.  Thus if $H$ and $K$ are pure subgroups of $G$ then $H\dot=K$ if and only if $H=K$. Clearly quasi--isomorphic groups have the same rank.

$G$ is \textit{quasi--decomposable} if $G\dot =H\oplus K$ for non--zero $H$ and $K\in\calG$, and $H$ and $K$ are called \text{quasi--summands}.
$G$ is \textit{strongly indecomposable } (si) if it has no proper quasi--decomposition. Equivalently, $G$ is si if and only if for all $k\in\bbN,\ kG$ is indecomposable. 
A direct sum $\oplus_{i\in[t]} H_i$ with each $H_i$ si, is called a \textit{strong decomposition}. It follows from \cite[Theorem 9.9]{Fuchs} that strong decompositions are unique up to isomorphism, i.e., if $\oplus_i H_i=\oplus_j K_j$ are strong decompositions, then their summands can be paired so that $H_i\cong K_j$.

 It was shown in \cite[Section 6]{Schultz2020} that every non--zero $G\in\calG$ contains   as a subgroup of finite index a strong decomposition $\oplus_{i\in[t]} H_i$  in which each $H_i$ is pure in $G$.  
 We call such a subgroup a   \textit{\Jo subgroup of $G$}, denoted $J(G)$. By the finite index property, $G=(J(G))_*$.
  For example, by \cite[Chapter 4]{M2000}, if $G$ is an acd group, then regulating subgroups are \Jo subgroups.
  
Clearly rank 1 groups are  si and   it was shown in \cite[Theorem 7.2]{Schultz2020} that $J\in\calG$ of rank $>1$ is si if and only if $J$ is an extension of a cd group $H$ by an infinite torsion group $T$ such that $J/H$ has no decomposition lifting to $J$.  

Extending the well known notion of type (\cite[\S2.2]{M2000}) from rank 1 groups to si groups in general, the  isomorphism class of an si group is called a \textit{Type}, so types are  special cases of Types. Denote by $\calT(D)$ the set of Types of a strong decomposition   $D$. By Lady's Theorem \cite[Corollary 6.5]{Schultz2020}, $\calT(D)$ is finite and a strong decomposition has the form $D=\bigoplus_{J\in\calT} J^{n_J}$ where $\calT=\calT(D)$ and each $n_J\in\bbN$.  Such a decomposition, which is unique up to isomorphism,  is  called a \textit{homogeneous decomposition} of $D$, and the summands $ J^{n_J}$ are called the \textit{homogeneous parts} of $D$.

  Near isomorphism is an equivalence on $\calG$ weaker than isomorphism but stronger than quasi--isomorphism. There are several equivalent characterisations, see for example \cite[\S 9.1]{M2000}, of which the most useful in our case is that $G$ is nearly isomorphic to $H$ if there exists $K\in\calG$ such that $K\oplus G\cong K\oplus H$.

We define an ostensibly stronger equivalence by   defining groups $G$ and $H$ to be 
\textit{strongly stably isomorphic}, denoted $G\cong_{ss}H$, if there exists a strong decomposition $K$ such that $K\oplus G\cong K\oplus H$.  It is known, \cite[Corollary 9.1.9]{M2000}, that for acd groups, near and strongly stable isomorphism are equivalent; however, it is still undetermined whether there exist nearly isomorphic groups $G$ and $H$ which fail to be strongly stably isomorphic.

A subgroup $H$ of $G$ is \textit{full} if $\rank(H)=\rank(G)$; equivalently, $H$ is full in $G$ if and only if  $G/H$ is torsion, and hence if and only if $H_*=G$. In particular, $J(G)$ is full in $G$.

Unexplained notation comes from the standard references \cite{Fuchs} or\hfill\break \cite{M2000}.

\section{The Principal Decomposition} Let $H$ be a strong decomposition with Type set $\calJ$,  so that $H$ has a homogeneous decomposition $\bigoplus_{J\in\calJ}J^{n_J}$.   By the uniqueness property of strong decompositions, each decomposition $H=L\oplus K$    has the form
$L=\bigoplus_{J\in\calJ}J^{m_J}$ and $K =\bigoplus_{J\in\calJ}J^{\ell_J} $ where $m_J$ and $\ell_J$ are non--negative integers such that $m_J+\ell_J=n_J$.
  
Recall that  a \textit{Principal Decomposition} of $G\in \calG$ is a decomposition $G=G_{sd}\oplus G_{ni}$ in which $G_{sd}$ is a strong decomposition and $G_{ni}$ has no si summand.

\begprop \label{uniquenes} Let 
$G=D\oplus B=A\oplus C$ where $D$ and $A$ are   strong decompositions and $B$ and $C$ have no si summand. Then $D\cong A$ and $B\cong_{ss} C$.

\enprop
\begpr 
  Let $\pi_A$ and $\pi_C$ be the projections  corresponding to the second decomposition.  
Now   $D=D(\pi_A+\pi_C)=D\pi_A\oplus D\pi_C$ so $D=J(D)\cong J(D\pi_A)\oplus J(D\pi_C)$. Since $D$ is a strong decomposition while 
   $C$ contains no si summand, $D\cong  J(D\pi_A)$ is isomorphic to a summand of    $A$.  Similarly, $A$ is isomorphic to a summand of $D$. Since $G$ has finite rank, $A\cong D$ and hence by definition, $B\cong_{ss} C$.
\enpr

  \begin{theorem} Every $G\in\calG$ has a Principal Decomposition $G=G_{sd}\oplus G_{ni}$, in which  $G_{sd}$ is unique up to isomorphism, and $G_{ni}$ up to strongly stable isomorphism.

\end{theorem}
\begpr  The theorem is vacuously true if $G$ is a strong decomposition or  has no strongly indecomposable summands, so assume neither of these extreme cases occur.

Let $\calJ$ be the set of all strong decompositions which are  summands of $G$.    Let $G_{sd}$ be an element of $\calJ$ of maximum rank, and let  $ G_{ni} $ be any direct complement of $G_{sd}$. Then $G_{ni}$ has no strongly decomposable summand, so this is a Principal Decomposition. By Proposition \ref{uniquenes}, $G_{sd}$ is unique up to isomorphism, and $G_{ni}$ up to strongly stable isomorphism.
 \enpr
 \begin{corollary} All  summands of $G$ which are maximum rank strong decompositions are isomorphic.
 \qed \end{corollary} 
 
\begin{remark} Principal Decompositions do not respect direct decompositions, in the following sense:
  Let $G=H\oplus K\in\calG$. Then in general, $G_{sd}\ne H_{sd}\oplus K_{sd}$. For example there is a well known example \cite[Theorem 5.2]{Fuchs} of a rank 4 group $G$ having a  Principal Decomposition with $\rank(G_{sd})=1$ such that $G=H\oplus K$ where both $H$ and $K$ are indecomposable of rank 2.\end{remark}

\section{ Ranges and Torsion Quotients}

We call a maximal independent subset of $G$ a \textit{basis}.  Bases of groups were studied in \cite[Proposition 2.2]{Schultz2020}, where it was shown that if $B$ is a basis of $G$ then every $0\ne a\in G$ has a unique representation  $a= k^{-1}\sum_{b\in B} n_bb$ where $n_b\in\bbZ,\ k\in\bbN$ and $\gcd\{k,\, n_b\colon b\in B\}=1$. 
 
A  \textit{range}  of $G$ is a full cd subgroup whose the rank 1 summands are pure in $G$. If  $B$ is a basis of $G$, the group $(B)_*=\bigoplus_{b\in B}b_*$ is a range,   and all ranges are of the form $(B)_*$ for some basis $B$. 

Since $(B)_*$ is full in $G$, $G/(B)_*$ is a torsion group, called the \textit{torsion quotient of $G$ with respect to $(B)_*$}. 
\beglem With the notation above,  $(B)_*=J(G)$ if and only if $(B)_*$ has finite index in $G$.

\enlem

\begpr If $(B)_*=J(G)$, then by definition, $(B)_*$ has finite index in $G$. Conversely, since $(B)_*$ is a strong decomposition  each of whose summands is pure in $G$, if   
$(B)_*$ has finite index in $G$ then $(B)_*=J(G)$. 
\enpr 

 \begprop \label{let J} Let $G\in\calG$  have basis $B$ such that $G/(B)_*$ is infinite and reduced. Then for every basis $C,\ G/(C)_*$ is infinite and reduced. 
 \enprop
\begpr  Let $\{x_i+(B)_*\colon i\in\bbN\}$ be an infinite independent set in  $G/(B)_*$. Suppose by way of contradiction that $\{x_i+(C)_*\colon i\in\bbN\}$ is dependent in   $G/(C)_*$, say $\sum_{j\in [n]}r_jx_j=c\in(C)_*$. Let $c$ have $B$--representation $k^{-1}\sum_{i\in[t]}n_ib_i$. Then $\sum_{j\in [n]}r_jx_j\in(B)_*$, contradicting independence of $\{x_i+(B)_*\}$.

Now suppose   that $G/(C)_*$ has a divisible summand $\cong \bbZ(p^\infty)$ for some prime $p$.  Then there exist $\{y_i\colon i\in\bbN\}$ in $G$ such that $py_0\in (C)_*$ and $py_{i+1}-y_i\in(C)_*$.   Each $y_i$ has $B$--representation $z_i=k_i^{-1}\sum_{j\in[t_i]}n_{ij}b_{ij}$, so that $\{z_0,\, z_i\colon i\in\bbN\}$  generates a summand of $G/(B)_*\cong \bbZ(p^\infty)$.  
This contradiction implies that $G/(C)_*$ is reduced.
 \enpr
 
 \begin{corollary} \label{basisinv}Let $G$ have torsion quotient $T=G/(B)_*$. Then the property that $T$ is finite or has a divisible summand is independent of the choice of basis $B$.
 \qed \end{corollary}

 A useful property relating decompositions of $G$ to decompositions of $J(G)$,  proved in \cite[Proposition 6.4]{Schultz2020}, is the following:
 \begin{proposition}\label{useful} Let $G\in\calG$.
 \begin{enumerate}\item If $G=H\oplus K$, then $J(G)=J(H)\oplus J(K)$.
 
 \item If $J(G)= L\oplus M$, then $G=L_*\oplus M_*$ if and only if $L_*\oplus M_*$ is pure in $G$.
 
\qed \end{enumerate}\end{proposition}

\begprop\label{letJ} Let $J$ be a strongly indecomposable  group. Then either $J$ has rank 1, or  for every basis $B$ of $J,\ J/(B)_*$ is   infinite and has no decomposition lifting to $J$. 
\enprop

\begpr Rank 1 groups are evidently si, so assume $\rank(J)>1$.  Let $B$ be a basis of $J$, so $(B)_*$ is a full cd subgroup of $J$ and hence $T=J/(B)_*$ is torsion. If $T$ is finite, then $J\dot= (B)_*$  so $\rank(J)=1$.

Suppose then that $T$ is infinite. If $T$ has no decomposition lifting to $J$, then by \cite[Theorem 7.2]{Schultz2020}, $J$ is 
si, while if $T$ has   such a decomposition, then by definition, $J$ is decomposable, a contradiction.
\enpr

 \section{The Torsion Quotient Decomposition} 
 We shall use the following notation:
 
\begin{itemize}\item
  $\calJ\colon $   the set of  strongly indecomposable groups in $\calG$, and $\oplus\calJ\colon $ finite direct sums from $\calJ$, i.e., strong decompositions;
\item $\calJ^{(fq)}\colon$ the subset of strong decompositions in $\calJ$ having finite torsion quotient, and $\oplus\calJ^{(fq)}\colon $ finite direct sums from $\calJ$;
\item $\calJ^{(rq)}\colon$ the subset of  strong decompositions in $\calJ$ having reduced torsion quotient, and $\oplus\calJ^{(rq)}\colon $ finite direct sums from $\calJ$.
\end{itemize} 
  
 \begin{proposition}\label{hierarchy} The three sets in the  hierarchy of subclasses 
 \[\oplus\calJ^{(fq)}\subset\oplus\calJ^{(rq)} \subset\oplus \calJ,\]   are closed under isomorphism,   direct sums and direct summands.
  \end{proposition}

\begpr  Let $G\in\oplus\calJ$ and $\theta\colon G\to H$ an isomorphism.  Then by \cite[Section 6]{Schultz2020} $H\in\oplus\calJ$.  If $G\in\oplus\calJ^{(fq)}$ or  $\oplus \calJ^{(rq)}$ then by Lemma \ref{let J}, $H\in\oplus\calJ^{(fq)}$  or $\oplus \calJ^{(rq)}$ respectively.

The closure of each of the sets under direct sums  follows directly from the definitions.  

Let $H=K\oplus L\in\oplus \calJ$.  Then $H$ has a basis $B=C\cup D$ in which $C$ is a basis of $K$ and $D$ a basis of $L$. Hence $H/(B)_*\cong K/(C)_*\oplus L/(D)_*$ and $K\in\oplus\calJ$. Furthermore,  $K\in\oplus \calJ^{(fq)}$ or $\oplus\calJ^{(rq)}$ if   $H$ is.
\enpr
 
To  show that  we have a proper hierarchy, we introduce the following notation: 

Let $\bbP=\{P_0,\cdots, P_n\}$ be a partition of the set of primes into $n+1$ infinite sets. For $i\in [n]$, let $\tau_i$ be the type which is infinite for all $p\in P_i$ and zero elsewhere. Let $\{v_i\colon i\in[n]\}$ be independent elements in $V$, and $U=\bigoplus_i [v_i]\cong \bbQ^n$.

 For all $i\in[n]$, let $b_i=\tau_i v_i,\  (B)_*=\bigoplus_i b_{i*}$. and $T=U/(B)_*$. 

\begprop For every $n\in\bbN$ there exist $H$ in $ \calJ^{(rq)}\setminus \calJ^{(fq)}$ and  $K\in \calJ\setminus  \calJ^{(rq)}$, each of rank $n$.

\enprop

\begpr With the notation above,  $T$ contains a summand isomorphic to $ \bigoplus_{p\in P_0}\bbZ(p^\infty)$ and hence a subgroup $A\cong \bigoplus_{p\in P_0}\bbZ(p)$ and for each $q\in P_0$  a summand $A_q\cong \bbZ(q^\infty)$. Note that no element of $(B)_*$ is divisible by all $p\in P_0$ and no element has infinite $q$--height.

Let $H= A\eta^{-1}$ and $K_q=A_q\eta^{-1}$. Then by Proposition \ref{useful} (2), $H$ and $K_q$ are strongly indecomposable with range $(B)_*$ and $H/(B)_*\cong A\in \calJ^{rq}\setminus \calJ^{fq}$ and $K_q/(B)_*\in \calJ \setminus \calJ^{rq}$ as required.
\enpr
 We  use the hierarchy to define the Torsion Quotient Decomposition. Let 
 \begin{itemize}\item $\calG^{(fq)}=\{G\in\calG\colon J(G)\in\oplus\calJ^{(fq)}\}$

\item $\calG^{(rq)}=\{G\in\calG\colon J(G)\in\oplus\calJ^{(rq)}\setminus\oplus\calJ^{(fq)}\}$
 
 \item $\calG^{(dq)}=\{G\in\calG\colon J(G)\in\oplus\calJ\setminus\oplus\calJ^{(rq)}\}$
\end{itemize}

Note that    $\calG^{(fq)}$ is the set of acd groups and all torsion quotients of each $G\in\calG^{(dq)}$ contain a divisible summand.

\begprop\label{5.4}   The sets $\calG^{(fq)},\ \calG^{(rq)}$ and  
 $\calG^{(dq)}$ form a partition of $\calG$ and are closed under     isomorphism,  direct sums and  direct summands;
\enprop 
 \begpr  That they form a partition of $\calG$ is clear by definition, and their closure properties follow from Proposition \ref{useful} (1) and  Proposition \ref{hierarchy}.   
   \enpr

 A \textit{Torsion Quotient Decomposition of $G$} is a decomposition \[G=G^{(fq)}\oplus G^{(rq)}\oplus G^{(dq)}\] where $G^{(fq)}\in\calG^{(fq)},\  G^{(rq)}\in\calG^{(rq)}$ and $G^{(dq)}\in\calG^{(dq)}$.

 \begin{theorem}\label{principal} Let $G\in\calG$. Then $G$ has a Torsion Quotient Decomposition. 
 
 If $G=G^{(fq)}\oplus G^{(rq)}\oplus G^{(dq)}  = H^{(fq)}\oplus H^{(rq)}\oplus H^{dq)}$ are Torsion Quotient Decompositions of $G$, then  $G^{(fq)}\cong H^{(fq)},\  G^{rq)}  \cong H^{(rq)}$ and $G^{(dq)}\cong H^{(dq)}$.  \end{theorem}
\begpr 
 Let $G\in\calG$. Let $G^{(fq)}$ be a $\calG^{(fq)}$ summand of $G$ of maximal rank, and let $D$ be any complementary summand, so that $D\in \calG^{(rq)}\cup\calG^{(dq)}$; let $G^{(rq)}$ be a $\calG^{rq)}$ summand of $D$ of maximal rank, and let $G^{(dq)}$ be any complementary summand of $G^{(fq)}\oplus\calG^{(rq)}$, so that $G^{(dq)}\in\calG^{(dq)}$ and  $G=G^{(fq)}\oplus G^{(rq)}\oplus G^{(dq)}$.
 
 Since \Jo subgroups are unique up to isomorphism, $J(G)\cong J(G^{(fq)})\oplus J(G^{(rq)})\oplus J(G^{(dq)})\cong J(H^{(fq)})\oplus J(H^{(rq)})\oplus J(H^{(dq)})$ with   $J(G^{(fq)})$ and $J(H^{(fq)})\in\oplus\calG^{(fq)},\ J(G^{(rq)})$ and $J(H^{(rq)})\in\calG^{(rq)}$ and $J(G^{(dq)})$ and $J(H^{(dq)})\in\calG^{(dq)}$.
 
 Thus $J(G^{(fq)})\cong J(H^{(fq)}) $ and hence $G^{(fq)}\approx H^{(fq)}$. Since summands of  $G$ are pure,     $G^{(fq)}\cong H^{(fq)}$, and similarly for the other summands, 
 $ G^{(rq)}  \cong H^{(rq)}$ and $G^{(dq)}\cong H^{(dq)}$.  
 \enpr
   \begin{remark}  It follows from the definitions that the Torsion Quotient Decomposition of $G$ respects direct sums, in the sense that if $G=H\oplus K$ then $G^{(fq)}=H^{(fq)}\oplus K^{(fq)},\ G^{(rq)}=H^{(rq)}\oplus K^{(rq)}$ and $G^{(dq)}=H^{(dq)}\oplus K^{(dq)}$ 
 \end{remark}
\subsection{The Principal and Torsion Quotient Decompositions are Compatible}
Since the decompositions are independent of each other,  both can be applied, to achieve canonical decompositions with 6 terms.  We show that these terms are independent of the order in which the decompositions are applied.

\begprop Let $G\in\calG$ and let $*\in\{fq,\,rq,\,dq\}$. Then $(G_{sd})^{(*)}=(G^{(*)})_{sd}$ and $(G_{ni})^{(*)}=(G^{(*)})_{ni}$

\enprop

\begpr The subgroup $G_{sd}\cap G^{(*)}$ of $G$ can be considered as both a maximal summand of $G_{sd}$ contained in $\calG^{(*)}$, i.e., $(G_{sd})^{(*)}$ and a maximal summand of $G^{(*)}$ contained in $\calG_{sd}$, i.e.
$(G^{(*)})_{sd}$.

Consequently, $G_{ni}\cap G^{(*)}$   can be considered as both a complement of a maximal summand of $G_{sd}$ contained in $\calG^{(*)}$, i.e., $(G_{ni})^{(*)}$ and  a complement of a maximal summand of $G^{(*)}$ contained in $\calG_{sd}$, i.e.
$(G^*)_{ni}$.
\enpr

\section{Ranked Atomic Categories}

In the preceding development, we used no properties of $\calG$ other than the fact that  it forms a \textit{ranked atomic category} in the following sense:

Let $\calC$ be an additive category having a set of indecomposable objects $\calD$.
\begin{enumerate}\item $\calC$  is
 \textit{atomic} if  every object is a coproduct of finitely many  objects in $\calD$;
 \item   $\calC$ is \textit{ranked} if there is  a function $\rank\colon \text{objects of }\calC\to \bbN$ satisfying 
 \begin{itemize}\item$\rank(H)=0$ if and only if $H=\{0\}$,
 \item $H\cong K$ implies $\rank(H)=\rank(K)$, and 
 \item$\rank(H\oplus K)=\rank(H)+\rank(K)$.
 \end{itemize}
\end{enumerate}

The category of finite rank modules over an integral domain, and more generally categories  of  modules with a composition series  over an arbitrary ring,    are examples of ranked atomic categories.

 Let $\calC$ be a ranked atomic category, and   $\calD$    the set of indecomposable objects of $\calC$. For any subset $\calA$ of $\calD$ let $\Add(\calA)$ be the closure of $\calA$ under isomorphism, direct sums and summands.

A partition $\{\calA,\,\calB\}$  of $\calD$ is called \textit{summand disjoint} if  no non--zero  element of $\Add(\calA)$ is isomorphic to   an element of $\Add(\calB)$ and \textit{vice versa}.

\begin{proposition}\label{atomic}
Let $\calC$ be a ranked atomic category and $\{\calA,\,\calB\}$ a summand disjoint partition of the indecomposable objects of $\calC$. 
 Then \begin{enumerate} \item every object $G$ of $\calC$ is a direct sum $G=A\oplus B$ with $A\in \Add(\calA)$   and $B\in \Add(\calB)$; 
\item if also $G=A'\oplus B'$ with $A'\in \Add(\calA)$   and $B'\in \Add(\calB)$, then    $A\cong A'$  and $B\cong B'$.
 \end{enumerate}
\end{proposition}

\begin{proof}  (1) This is clear if $G\in\Add(\calA)$ or $\Add(\calB)$, so we may assume $G$ has summands from both $\Add(\calA)$ and $\Add(\calB)$. Let $\calC$ be the set of all  summands of  $G$  which are elements of $\Add(\calA)$ and let $A$ be an element of $\calC$  of maximal rank, so 
$G=A\oplus B$ where $0\ne B$ has no summand isomorphic to an element of  $ \Add(\calA)$. Then every indecomposable summand of $B$ is in $\calB$, so $B\in\Add( \calB)$.

(2) Each of $A,\,A',\,B$ and $B'$ is a direct sum of indecomposables.  Let $C$ be an indecomposable summand of $A$, so $C\in\calA$.  Now $C\cong D\oplus E$ with $D\leq A'$ and $E\leq B'$. Since $C$ is indecomposable and $B'$ has no summand from $\calA,\ C\cong D$. Thus every indecomposable summand of $A$, and hence $A$ itself is isomorphic to a summand of $A'$. 
 
By a symmetrical argument, $A'$ is isomorphic to a summand of $A$; and similarly, $B$ is isomorphic to a summand of $B'$ and $B'$ to a summand of $B$.
  Since rank is preserved by isomorphism,   $A\cong A'$ and $B\cong B'$.
\end{proof}

 It is straightforward to extend the definition of summand disjoint partition to partitions with more than two parts and to prove the corresponding analogues of Proposition \ref{atomic} of such partitions. 
Thus we conclude:

\begin{theorem} \label{main} Let $\calC$ be a ranked atomic category   and $\{\calA_i\colon i\in[n]\}$   a summand disjoint partition of its indecomposable objects. Then every  object $G$ of $\calC$ has a decomposition $G=\bigoplus_{i\in[n]}G_i$ with $G_i\in\Add(\calA_i)$ 
   which is  unique up to  isomorphism.
\end{theorem}

\begin{proof}

Both existence and uniqueness  follow by a routine induction from   Proposition \ref{atomic}.
 \end{proof}

 \end{document}